\documentclass[12pt]{amsart}
\usepackage{amsmath}
\usepackage{amssymb}
\usepackage[mathcal]{eucal}
\usepackage[all]{xy}
\usepackage{latexsym}
\usepackage{amstext}
\usepackage{amsfonts}
\usepackage{amsthm}
\usepackage{amsopn}
\usepackage{amsbsy}
\usepackage{layout}
\usepackage{color}
\usepackage{graphicx}
\usepackage{bm}

\theoremstyle{plain}
\newtheorem{theorem}{Theorem}[section]
\newtheorem{proposition}[theorem]{Proposition}
\newtheorem{lemma}[theorem]{Lemma}

\theoremstyle{definition}
\newtheorem{definition}[theorem]{Definition}
\newtheorem{problem}[theorem]{Problem}
\theoremstyle{remark}

\newtheorem{remark}[theorem]{Remark}

\def\nn{\|\cdot\|}

\def\ee{\varepsilon}

\def\xx{\overline x}

\def\N{{\mathbb N}}
\def\R{{\mathbb R}}
\def\B{{\mathbb B}}

\def\vv{\varphi}
\def\ee{\varepsilon}

\def\ee{\varepsilon}

\def\nn{\|\cdot\|}

\date{}

\begin{document}
\title[Asplund spaces $C_k(X)$ beyond Banach spaces]{Asplund spaces $C_k(X)$ beyond Banach spaces}
\author{Marian Fabian}
\address{Institute of Mathematics, Czech Academy of Sciences, Prague, Czech Republic}
\email{fabian@math.cas.cz}
\author{Jerzy K{\c{a}}kol}
\address{Faculty of Mathematics and Informatics, A. Mickiewicz University,
61-614 Pozna\'{n}, Poland}
\email{kakol@amu.edu.pl}
\author{Arkady Leiderman}
\address{Department of Mathematics, Ben-Gurion University of the Negev, Beer Sheva, P.O.B. 653, Israel}
\email{arkady@math.bgu.ac.il}

\begin{abstract}
This paper addresses the Asplund property for the space of continuous functions $C_k(X)$ equipped with the compact-open topology, when $X$ is an arbitrary Tychonoff space.  Motivated by inconsistent definitions in prior literature extending the Asplund property beyond Banach spaces, we provide a unified and self-contained treatment of core results in this context.

 A characterization of the Asplund property for $C_k(X)$ is established, alongside a review of classical results, including the Namioka--Phelps theorem and its implications. All proofs are presented in a self-contained manner and rely on standard techniques.
\end{abstract}

\bigskip

\maketitle
\section{Introduction}

In his seminal book \cite{p}, R.R. Phelps observed that the first infinite-dimensional result concerning the differentiability properties of convex functions was obtained in 1933 by S. Mazur  \cite{Mazur}:
A continuous convex function $f:D \to \R$ defined on an open convex subset $D$ of a separable Banach space $E$, is G\^ ateaux differentiable
at the points of a dense $G_{\delta}$ subset of $D$.

In 1968, E. Asplund \cite{Asplund} extended Mazur's theorem in two ways. Namely, he proved the same statement for a more general class of Banach spaces;
and studied a more restricted class of Banach spaces (now called Asplund spaces) in which a
stronger conclusion holds: Fr\'echet differentiability on a dense $G_{\delta}$ subset of $D$.

Over the years, this area has attracted considerable attention of several  researchers. In our recent work \cite{KL3}, we noted that many contributions, especially those extending beyond the framework of Banach spaces, have not been consistently or adequately cited in the literature. This observation has motivated us to provide a more uniform account of relevant results and to present a comprehensive extension of the celebrated Namioka--Phelps theorem in the context of the spaces of continuous functions $C_k(X)$, now encompassing arbitrary Tychonoff spaces $X$.

In this paper all topological spaces are assumed to be Tychonoff.
They can be can be described as subspaces of the compact cubes $[0,1]^{T}$.

For a Tychonoff topological space $X$, we denote by
$C_k(X)$ the vector space of all continuous real-valued functions on $X$, equipped with the {\it compact-open topology},
 whose base consists of the sets
$$
V^\ee_K:= \{f\in C_k(X):\ f[K]\subset (-\ee,\ee)\}
$$
where $K$'s run throughout all nonempty compact subsets of $X$ and $\ee$'s run throughout all positive numbers; thus $C_k(X)$ becomes a locally convex space.

We use the abbreviation lcs for locally convex spaces. 
Let us begin by clarifying the notions of openness, continuity, and the Lipschitz property of mappings in the setting of lcs.

\begin{proposition}\label{ocl}
Let $X$ be a lcs, $D\subset X$,  $0\in V\subset X$ convex open sets,  and $\vv: X\to \R$ a convex function
The following assertions are equivalent:

(1) $\vv$ is $V$-bounded above, at some (every) $x\in D$, i.e., there is a real $t > 0$ such that
 $\sup\vv[(x+tV)\cap D]<\infty$.


(2) $\vv$ is $V$-continuous at some(every) $x\in D$, i.e. for every $\ee>0$ there is a $t>0$ such that $|\vv(x')-\vv(x)|<\ee$ whenever $x'\in (x+tV)\cap D$.

(3) $\vv$ is locally $V$-Lipschitzian, i.e., for every $x\in D$ there are $t>0$ and $c>0$ such that $\vv(x_1)-\vv(x_2)\le c\,p_V(x_1-x_2)$ whenever $x_1,x_2 \in (x+tV)\cap D$,
where $p_V$ is the Minkowski functional of $V$.
\end{proposition}

\noindent The proof is left to a reader as an exercise.

We will use some basic facts about scattered topological spaces, see for example \cite[Section 8.5]{Semadeni}.
A topological space $X$ is said to be {\it scattered} if every (closed) nonempty
subset $S$ of $X$ has an isolated point in $S$. 
Denote by $A^{(1)}$  the set of all non-isolated (in $A$) points of $A \subset X$. 
For ordinal numbers $\alpha$, the $\alpha$-th {\it Cantor--Bendixson derivative} of a topological space $X$
is defined by transfinite induction as follows. 
\begin{itemize}
\item  $X^{(0)} = X$;
\item  $X^{(\gamma)} = (X^{(\gamma-1)})^{(1)}$ \text{if $\gamma>0$ is a non-limit ordinal};
\item  $X^{(\gamma)} = \bigcap_{\alpha<\gamma} X^{(\alpha)}$ \text{if $\gamma$ is a limit ordinal}.
\end{itemize}
Clearly,  $X$ is scattered if and only if  $X^{(\alpha)}$ is empty for some ordinal $\alpha$.
For a compact space $X$ this implies the following easy assertion.

\begin{lemma} \label{scattered} 
A compact space $X$ is scattered if and only if there is no nonempty closed perfect subset of $K$.
\end{lemma}

Next important results characterize scattered compact spaces in terms of continuous images.

\begin{theorem}\label{sca}\hfill
 \begin{enumerate}
\item {\rm \cite[Proposition 8.5.3]{Semadeni}}
Let $\pi: X \twoheadrightarrow Y$ be a continuous surjection, where $Y$ is a Hausdorff space.
If $X$ is a scattered compact space then so is $Y$.
\item {\rm \cite[Theorem 8.5.4]{Semadeni}}
A compact space $X$ is scattered if and only if there is no  continuous surjection from $X$ onto the closed interval $[0,1]$.
\end{enumerate}
\end{theorem}

We will use also the following classical theorem of S. Mazurkiewicz and W. Sierpi\'nski.

\begin{theorem}\cite[Theorem 8.6.10]{Semadeni}\label{ordinal}
Every compact scattered first-countable space (in particular, every metrizable compact space) is homeomorphic to a countable compact ordinal.
\end{theorem}

In several places, we  use the following applicable result,  whose proof follows from the classical Tietze--Urysohn extension theorem;
see for example \cite[Extension Lemma, p. 29]{Jarchow} or \cite[Section 3.6]{e}.

\begin{proposition}\label{T}
Let $X$ be a  Tychonoff space and let $L\subset X$ be a compact subset of $X$. Then every continuous function $f: L\rightarrow [0,1]$ can  be extended to a continuous function $F:X\rightarrow [0,1]$  defined on the whole space $X$.
\end{proposition}

Let $Z$ be a topological linear space and let $\vv: Z\to\R$ be a function, defined in a vicinity of a fixed point $z \in Z$.
We say that $\vv$ is {\it Fr\'echet differentiable} at $z$ if it is continuous at $z$ and there exists a continuous linear functional $\xi:Z\to\R$ such that
$$
\sup_{h\in M}\big( \vv(z+th)-\vv(z)-\xi(th)\big) =o(t)\quad {\rm as}\quad \R\ni t\to 0
$$
for every bounded set $M\subset Z$ (where $o(\cdot)$ is a suitable function, defined around $0$, such that
$\frac{o(t)}t\to0$ as $t\to0$).
This definition can be found in, say \cite{s}, \cite{y}. Let us temporarily call the function $\vv$ {\it Yamamuro--F.d}.

We say that $\vv$ is {\it Fr\'echet differentiable} at $z$ if it is continuous at $z$ and there exist a continuous linear functional $\xi:Z\to\R$ and a neighbourhood $V$ of $0\in Z$
such that
$$
\sup_{h\in V}\big( \vv(z+th)-\vv(z)-\xi(th)\big) =o(t)\quad {\rm as}\quad \R\ni t\to 0.
$$
This definition can be found in, say \cite[p. 11]{sc}. Let us temporarily call the function $\vv$ {\it Schwartz--F.d}.
In both versions, the (unique) $\xi$ is denoted as $\vv'(z)$ and is called {\it Fr\'echet derivative} of $\vv$ at the point $z$.

For a much detailed and profound account of the differentiability in linear topological spaces we recommend the survey article by V.I. Averbukh and O.G. Smolyanov \cite{as}.
   
Obviously, Schwartz--F.d. is stronger than  Yamamuro--F.d. However, the notions are equivalent in some important cases.
First, if  $Z:=C_k(X)$, where $X$ is any Tychonoff space, see Proposition \ref{SY} below; and, second, if $Z:=(E,w)$ or  $Z:=(E^*,w^*)$, where $E$ is any Banach space, see \cite{s}. 
Of course, the convexity of $\vv$ plays a crucial role here. 

\begin{proposition}\label{SY} Let $X$ be a Tychonoff topological space, let $D\subset C_k(X)$ be a convex open set, let $\vv: D\to\R$ be a convex continuous function, and let $f\in D$ be given.

\noindent Then $\vv$ is Schwartz--Fr\'echet differentiable at the point $f$ if and only if $\vv$ is Yamamuro--Fr\'echet differentiable at $f$.
\end{proposition}
\begin{proof}
 Put
\begin{equation}\label{M}
M:=  \{h\in C_k(X):\ h[X]\subset (-1,1)\}.
\end{equation}
Clearly, $M$ is a bounded set. Using Proposition~\ref{ocl}, we find a compact set $K\subset X$ and a $\gamma>0$
such that
$$
f+V^\gamma_K\subset D\quad {\rm and}\quad \sup\vv [f+V^\gamma_K]<\infty,
$$
where
$$
V^\gamma_K:= \{h\in C_k(X):\ h[K] \subset (-\gamma,\gamma)\};
$$
recall that this $V^\gamma_K$ is a convex open neighbourhood of $0\in C_k(X)$. (This $V^\gamma_K$ will witness for the
Schwartz--Fr\'echet differentiability of $\vv$ at $f$).

As $\vv$ is Yamamuro--Fr\'echet differentiable at $f$, we have that
$$
\sup_{h\in M}\big(\vv(f+th)-\vv(f)-\vv'(f)(th)\big) = o(t)\quad{\rm as}\quad t\to0.
$$
Hence
$$
\sup_{h\in M}\sum\vv(f\pm th)-2\vv(f) = o(t)\quad{\rm as}\quad t\downarrow0.
$$
Now, consider any $h\in V^\gamma_K$. Then $h[K]\subset (-\gamma,\gamma)$.
Sine $X$ is a Tychonoff space and $K\subset X$ is compact, Proposition~\ref{T} yields a continuous
extension $ \check h: X\to (-\gamma,\gamma)$ such that $\check h|_K=h|_K$.
By the Claim below, we get that $\vv(f+th)=\vv(f+t\check h)$, and therefore,
$$
\sup_{h\in V^\gamma_K}\sum\vv(f\pm th)-2\vv(f) = o(t)\quad{\rm as}\quad t\downarrow0.
$$
We also know that $\lim_{t\to0}\frac1t(\vv(f+th)-\vv(f))=\vv'(f)h$ for every $h\in C_k(X)$.
Thus, using the convexity of $\vv$, we can conclude that

$$
\sup_{h\in V^\gamma_K}\big(\vv(f+th)-\vv(f)-\vv'(f)(th)\big) = o(t)\quad{\rm as}\quad t\to0.
$$
We proved that $\vv$ is Schwartz--Fr\'echet differentiable at the point $f$, where the neighborhood $V^\gamma_K$ is a witness for that.

It remains to formulate and prove the announced claim/observation, going back to B. Sharp \cite[p. 208, 209]{s}:
\medskip

\noindent {\tt Claim.}  {\it If $f,f'\in D$ and $f|_K=f'|_K$, then $\vv(f)=\vv(f')$.}
\smallskip

\noindent Proof of Claim. 
For all $\lambda>1$ we have $f+\lambda(f'-f)\in D  $ and
$$
f'=\frac{\lambda-1}\lambda f +\frac1\lambda(f+\lambda(f'-f));
$$
hence, by the convexity of $\vv$, we have
\begin{eqnarray*}
\vv(f') &\le& \frac{\lambda-1}\lambda \vv(f) +\frac1\lambda\vv(f+\lambda(f'-f)) \\
        &\le& \frac{\lambda-1}\lambda \vv(f) +\frac1\lambda\sup\vv[f+V_K^\gamma] \longrightarrow \vv(f)\ \ {\rm as}\ \ \lambda\uparrow\infty;
\end{eqnarray*}
Hence $\vv(f') \le \vv(f)$. And, by symmetry, $\vv(f) \le \vv(f')$.
\end{proof}

\noindent We remark that the Yamamuro--F.d. in the argument above could be replaced by the (yet weaker) `$M$--Fr\'echet differentiability'
where $M$ is defined in (\ref{M}) above.
\smallskip

\begin{remark}
In view of Proposition \ref{SY}, since throughout the paper we deal with the $C_k(X)$-spaces, we can and will use only the term 
{\it Fr\'echet differentiability} without any extra adjectives.
\end{remark}

Next, we follow the definitions proposed in \cite{es} and \cite{s}.

\begin{definition}\label{def3}
A lcs $E$ is called an {\it Asplund (weak Asplund) space} if every continuous convex function $f: D \to \R$, where $D\subset E$ is a nonempty open and convex set,
is Fr\'echet (G\^ ateaux, respectively) differentiable at the points of a dense $G_{\delta}$ subset of $D$.
Asplund (weak Asplund) spaces are abbreviated by {\it ASP (WASP}, respectively).
\end{definition}

\begin{definition}\label{def4}
A lcs $E$ is called a {\it Fr\'echet (G\^ ateaux) Differentiability Space} if every continuous convex function $f: D \to \R$, where  $D\subset E$ is a nonempty open and convex set,
is Fr\'echet (G\^ ateaux, respectively) differentiable at the points of a dense subset of $D$.
Fr\'echet (G\^ ateaux) differentiable spaces are abbreviated by {\it FDS {\rm (}GDS}, respectively).
\end{definition}

For the  next important statement, due to J. Rainwater, we refer to \cite[Proposition 1.25]{p}.

\begin{proposition}\label{Phelps}
For any Banach space and any open set $D\subset X$, the (possibly empty) set of points of  Fr\'echet differentiability of any continuous convex function from $D$ into $\R$ is a $G_{\delta}$.
\end{proposition}

\noindent Consequently, the classes of Banach ASP and Banach FDS spaces coincide.

The classification of lcs according to the dense or generic differentiability
of convex continuous functions extends earlier  research  of E. Asplund \cite{Asplund},  D.G. Larman and R.R. Phelps \cite{Larman} and  I. Namioka and  R.R. Phelps \cite{np} which dealt with just Banach spaces.

M. Talagrand \cite{Tal1} constructed a 1-Lipschitz function on a Banach space $C(K)$, with some compact space $K$, whose set of points of G\^ ateaux differentiability is dense but of the first category.
 Moreover,  M. \v Coban and P. Kenderov \cite{Coban} provided  examples showing that even if the set
of G\^ ateaux differentiability points of the $\sup$-norm of a Banach space $C(K)$, with $K$ a suitable compact space, is dense, it need not contain a nonempty $G_{\delta}$ set.
For example, this situation occurs for the {\it double arrows} compact space $K$.
Note also that there  is a delicate example of a Banach space that is GDS but not weak Asplund \cite{Moors}.

One should mention that a systematic research of lcs  that are ASP, WASP, FDS, and GDS was started  in 1990 by B. Sharp \cite{s},
and then it was  continued in joint papers \cite{es}, \cite{Eyland_Sharp2},  \cite{Eyland_Sharp3}.
Regrettably, it appears that these important works have not received due attention in the later publications by several authors. 
In our paper \cite{KL3} we clarified this issue and also considered
the questions about the preservation of WASP and ASP for arbitrary products of locally convex spaces.



\section{Namioka--Phelps theorem and additional supplementing results}
Assume first that we can find points of Fr\'echet differentiability of convex continuous functions
defined on separable Banach spaces, that is, on those spaces $E$ which possess a countable dense subset.
Can we deduce from this ``separable'' information that the whole $f$, defined on $E$, is somewhere
Fr\'echet differentiable? The answer is affirmative and this is caused by availability of a suitable ``separable reduction'' of
the phenomenon of Fr\'echet differentiability.

Having in mind the fundamental property included in Proposition~ \ref{Phelps}, we propose also next Proposition \ref{111}.
For the concepts mentioned in the statement below we refer, for example, to \cite[p. 317--321]{y1} and
\cite[Chapter 11, p. 625--629]{y2}. In particular, once a function $f$ is concave and $\partial_F(x)$ is nonempty, then
$f$ must be Fr\'echet differentiable at $x$, with $\partial_Ff(x)=\{f'(x)\}$.

\begin{proposition}\label{111} Let $(E,\nn)$ be a (separable) Banach space whose dual $E^*$ is separable,
let $D\subset E$ be an open set, and and let $f:D\longrightarrow(-\infty,\infty]$ be a lower semi-continuous function.

\noindent Then the set of all $x\in D$, where the Fr\'echet subdifferential $\partial_F f(x)$
is nonempty, is dense in $D$.
\end{proposition}

\begin{proof} We first realize that $E$ admits an equivalent Fr\'echet smooth (off the origin) norm. Indeed,
let $\{x_1, x_2,\ldots\}$ be a countable dense subset of the unit ball $S_E$ of $E$ and $\{\xi_1, \xi_2, \ldots\}$ be a countable dense subset of the unit ball $S_{E^*}$ of the dual $E^*$. Then the assignment sending $x^*\in E^*$ to
$$
|x^*|=\sqrt{\|x^*\|^2 + \hbox{$\sum_{n=1}^\infty 2^{-n}$}\langle x^*,x_n\rangle^2+\hbox{$\sum_{n=1}^\infty 2^{-n}$}{\rm dist}\, (x^*,{\rm sp}\{\xi_1,\ldots,\xi_n\}\big)^2}
$$
is an equivalent weak$^*$ lower semi-continuous norm on $E^*$. For the reader's convenience we indicate why $|\cdot|$ is
subadditive, i.e., the inequality $|x^*+y^*| \le |x^*|+|y^*|$ holds. It is not easy to find in the literature how to prove this.
A possible argument may profit from the subadditivity of the famous (Euclidean) $\ell_2$-norm.

A `small' Kadec--Klee like effort, see \cite[p. 113-117]{d}, reveals that
the norm $|\cdot|$ is locally uniformly rotund, i.e. if $x^*, x^*_1,\ x^*_2,\ldots$ are elements of $E^*$ such that
$2|x^*|^2+2|x_i^*|^2-|x^*+x^*_i|^2\to0$ when $i\to\infty$, then $|x^*_i-x^*|\to 0$ when $i\to\infty$.
Assume this is not so. Find then an $\ee>0$ and an infinite set $N\subset \N$ such that $\|x^*_i-x^*\| > 3\ee$
for every $i\in N$.
Find $m\in N$ so big that dist$(x^*,{\rm sp}\{\xi_1,\ldots,\xi_m\}\big)<\ee$.
A simple convexity argument provides an infinite set $N_1\subset N$ such that
dist$(x^*_i,{\rm sp}\{\xi_1,\ldots,\xi_m\}\big)<\ee$ for every $i\in N_1$. For every such $i$  find
$y^*_i\in {\rm sp}\{\xi_1,\ldots,\xi_m\}$ such that $\|x^*_i-y^*_i\|<\ee$, and so,
$$
\|y^*_i-x^*\| \ge \|x^*_i-x^*\| - \|x^*_i-y^*_i\| > 3\ee -\ee=2\ee.
$$

The sequence $(y^*_i)_{i\in N_1}$ is bounded
because $\|x^*_i\|\to\|x^*\|$ as $i\to\infty$, by convexity. Further, since it lies in a finite-dimensional space, there are $y^*\in {\rm sp}\{\xi_1,\ldots,\xi_m\}$ and an infinite set
$N_2\subset N_1$ such that $\|y^*_i-y^*\|\to0$ as $N_2\ni i\to\infty$.
Hence $\|y^*-x^*\|\ge 2\ee.$ Find $n\in\N$ so that $\|\langle y^*-x^*,x_n\rangle\| >\ee$.
Another  simple convexity argument yields that $\langle x^*_i,x_{n}\rangle \to \langle x^*,x_{n}\rangle$ as $i\to\infty$.
Hence we have that 
$$
(\ee>)\ \| y^*_i-x^*_i\| \ge  \langle y^*_i-x^*_i, x_{n}\rangle> \ee
$$ 
for all $i\in N_2$ big enough, a contradiction. We proved that our norm $|\cdot|$ on $E^*$ is LUR.

Now, having proved that our norm $|\cdot|$ on $E^*$ is LUR, the \v Smulyan test guarantees that the predual norm $|\cdot|$ on $E$, defined by
$$
E\ni x\longmapsto\sup\big\{\langle x^*,x\rangle:\ x^*\in E^*,\ |x^*|\le1\big\}=:|x|,
$$
is Fr\'echet differentiable at every nonzero point of $E$. For more details, see \cite[p. 43]{dgz}.

Next, let $\overline x\in Df$ and $\ee>0$ be given. We will find a $v\in E$ such that
$|v-\overline x|<\ee$ and $\partial_F f(v)$ is nonempty.
From the lower semi-continuity of $f$ find $\ee'\in(0,\ee)$ so small that $B(\overline x,\ee')\subset D$
and that $f$ is bounded below on the open ball
$B(\overline x,\ee')$. Define
\begin{eqnarray*}
\vv(x):=\begin{cases}
          \big(\tan(\hbox{$\frac\pi{2\ee'}$} |x-\overline x|)\big)^2 & \text{if $x\in B(\xx,\ee')$}\\
          \infty & \text{if $x\in E\setminus B(\xx,\ee')$}.
        \end{cases}
\end{eqnarray*}
Then $\vv:E\longrightarrow[0,\infty]$ is easily seen to be proper and lower semi-continuous.
Now, the Borwein--Preiss variational principle \cite[Theorem 4.20]{p} provides a Fr\'echet differentiable function
$\theta: D\longrightarrow[0,\infty)$ such that the sum $f+\vv+\theta$ attains infimum at some $v\in D$; clearly,
$v\in B(\xx,\ee')$. We thus have
$$
f(v+h)+\vv(v+h)+\theta(v+h) \ge f(v)+\vv(v)+\theta(v)\quad{\rm for\ every}\quad h\in E,
$$
and so
$$
f(v+h)-f(v)+\langle \vv'(v)+\theta'(v),h\rangle \ge -o(|h|)\quad{\rm as}\ \ h\in E\ \ {\rm and}\  \ |h|\to0.
$$
We proved that $-\vv'(v)-\theta'(v)\in \partial_F f(v)$, and hence $\partial_F f(v)\neq\emptyset$.
\end{proof}
\begin{remark} Originally, when the Borwein--Preiss principle was not at hand, the argument above could be performed also with just Ekeland's principle.
\end{remark}
Now a main question arises: Is it possible to
extend Proposition~\ref{111} to non-separable Asplund spaces? The answer is affirmative.  The arguments  goes back to
D.A. Gregory's \cite[Theorem 2.14]{p}. (It should be noted that an analogous statement for Gateaux differentiability is
hopelessly not valid
---take the `limsup-norm' on the space $\ell_\infty$).

\begin{proposition}\label{g}
Let $(E,\nn)$ be a non-separable Banach space, $f:D\rightarrow\R$ be a  convex continuous function
defined on an open convex set $D\subset E$, and $Z$ be a separable subspace of $E$.

\noindent Then there exists a separable subspace $Y$ of $E$, containing $Z$, with $D\cap Y$ nonempty, and such that, if
the restriction $f|_Y$ of $f$ to $Y$ is Fr\'echet differentiable at some $x\in D\cap Y$, then  $f$ is
Fr\'echet differentiable at $x$.
\end{proposition}

\begin{proof}
First, we need a  translation
of Fr\'echet differentiability (of convex functions)  to the terms of the space $E$:
\sl $f$ is Fr\'echet differentiable at $x\in E$ if and only if
$$
S(x,t):=\sup_{h\in B_X}\big(f(x+th)+f(x-th)\big)=2f(x)+o(t)\ \ {\rm as}\ \ \ Q_+\ni t\downarrow 0;
$$
\rm this is easy to check. (Here and below, $Q_+$ means the set of all positive rational numbers). 
For any $x\in E$ and any $t>0$, if $S(x,t)<\infty$, we find a vector $u(x,t)\in B_E$ such that
\begin{eqnarray}\label{22}
f\big(x+ t\, u(x,t)\big) +f\big(x-t\,u(x,t)\big) > S(x,t)-t^2.
\end{eqnarray}
(This $u(x,t)$ is almost ``the worst possible'' regarding the
Fr\'echet differentiability of $f$ at $x$).
Let $C_0$ be a countable dense subset of $Z$. We  construct countable sets
$C_0\subset C_1\subset C_2\subset\cdots\subset E$ as follows. Let $m\in\mathbb{N}$ be given and assume
that $C_{m-1}$ was already found. Find a countable set $C_m$ in $E$ such that
it is stable under making all finite linear combinations with rational coefficients, and that
it contains $C_{m-1}$ as well as the set $\big\{u(x,t):\ x\in C_{m-1},\ t\in Q_{+}\}$;
clearly, $C_m$ is again countable.
Do so for every $m\in\mathbb{N}$, and put finally $Y:=\overline{C_1\cup C_2\cup\cdots}\,$.
Clearly, $Y\supset Z$.

We claim that this $Y$ has the desired property. So, assume that
$f|_Y$ is Fr\'echet differentiable at some $x\in Y$. We  show that the whole $f$ is
Fr\'echet differentiable at $x$ as well. (If $x\in C_1\cup C_2\cup\cdots$, then it is rather easy to proceed.
So, we have to find an argument working also for $x\in Y\setminus C_1\cup C_2\cup\cdots$).
Let $L$ denote a Lipschitz constant of $f$ in a vicinity of $x$; see \cite[Proposition 1.6]{p}.
Pick any $t\in Q_+$ small enough (so that we can profit below from the $L$-Lipschitz property of $f$ around $x$).
Find then $c\in\bigcup_{m=1}^\infty C_m$ such that $\|c-x\|<t^2$. We can now
subsequently estimate for all sufficiently small $t\in Q_{+}$ (so that we can use the
Lipschitz property of $f$)
\begin{eqnarray*}
2f(x)&\le& S(x,t) < S(c,t)+ 2Lt^{2}\\
&<& f\big(c+ t\,  u(c,t)\big)+f(c-t\, u(c,t)\big)+t^2 + 2Lt^{2}\\
&<& f\big(x+ t\,  u(c,t)\big)+f(x-t\, u(c,t)\big)+t^2 + 4Lt^{2}\\
&\le&   \sup_{k\in B_Y}\big(f(x+ t  k)+f(x-tk)\big)+t^{2} + 4 Lt^{2}\\
&=&o(t) + 2f(x)\quad {\rm as}\quad Q_{+}\!\backepsilon t\downarrow0
\end{eqnarray*}
since $f|_Y$ is Fr\'echet differentiable at $x$.
Therefore, the ``whole'' $f$ is Fr\'echet differentiable at $x$.
\end{proof}

In order to prove our main Theorem \ref{MAIN}  we will need  (and prove) several additional important results, especially, 
a remarkable theorem due to I. Namioka and R.R. Phelps, \cite[Lemma VI.8.3]{dgz}, or \cite[Theorem 14.25]{y2}, or \cite[Theorem 18]{np} for selected sources.
Another recommended comprehensive source of information is the book \cite{f}.
The proof of Theorem \ref{1} uses  some arguments presented in \cite[p. 626--627]{y2}, see also \cite[Main Theorem]{Pelczynski-Semadeni}, but it does  not use any  technology  related to differentiability.

\begin{theorem}\label{1} For a compact space $X$ the following assertions are equivalent:
 \begin{enumerate}
  \item $X$ is scattered.
  \item Every separable subspace of $C(X)$ has separable dual.
  \item The Banach space $C(X)$ does not contain a copy of $\ell_1$ isomorphically.
 \end{enumerate}
\end{theorem}

\begin{proof}
(1)$\Longrightarrow$(2) Let $Y\subset C(X)$ be any separable subspace. Choose a countable and linearly independent set $N\subset B_Y$ such that its linear span is dense in $Y$.

Consider a correspondence
$$
X\ni x\longmapsto \{n(x):\ n\in N\}=:\psi(x)\in [-1,1]^{N}.
$$
Thus $\psi: X\longrightarrow [-1,1]^{N}$, and this is a continuous mapping. Denote
$L:=\psi(X)$; this is a metrizable compact space.
And, as $X$ is scattered,  $L$ is scattered as well, by Theorem \ref{sca}. Moreover, $L$ is countable, by Theorem \ref{ordinal}.

Next, let us enumerate $L$ as $\{z_{n}:\ n\in N\}$.
It is known that every  linear continuous functional $x^*\in C(L)^{*}$ has the form
$$
x^*(f)=\sum_{n}a_n f(z_n),\quad f\in C(L),
$$
and $\|x^*\|=\sum_{n\in N} |a_n| < \infty$,
where $(a_n)_{n\in N}$ is a suitable element of $\ell_1(L)$, see \cite[Main Theorem (11)]{Pelczynski-Semadeni} or \cite[Corollary 19.7.7]{Semadeni}.
Consequently, $C(L)^{*}$ is isometric to $\ell_1(L)$, and this implies that $C(L)^*$ is separable.

We will show that the subspace $Y\subset C(X)$ is isometric to a subspace of $C(L)$.  
For $n\in N$ and $x\in X$ we put
$$
\tilde n(\psi(x)):= n(x).
$$
The $\tilde n$ is a well defined function on $L$. Indeed, if $\psi(x)=\vv(h)$ for some $x, h\in X$, then  $\tilde n(\psi(x))= n(x)=n(h)=\tilde n(\psi(h))$.

Fix any $n\in N$. The function $\tilde n$ belongs to $C(L)$. Indeed, consider a net $(\psi(x_\alpha))$ in $L$ converging to some $\psi(x)\in L$.
This means that $m(x_\alpha)\to m(x)$ for every $m\in N$. Thus, in particular, $n(x_\alpha)\to n(h)$. Therefore, $\tilde n(\psi(x_\alpha))=n(x_\alpha)\longrightarrow n(x)= \tilde n(\psi(x))$.

Now, for any element $(a_n)_{n\in N} \in c_{00}(N)$ we have
\begin{eqnarray*}
\Big\|\sum_{n\in N} a_n \tilde n\Big\|&=& \sup_{l\in L}\Big|\sum_{n\in N} a_n\tilde n (l)\Big| =
\sup_{x\in X}\Big|\sum_{n\in N} a_n\tilde n (\psi(x))\Big| \cr
&=& \sup_{x\in X}\Big|\sum_{n\in N} a_nn(x)\Big|= \Big\|\sum_{n\in N}a_nn\Big\|.
\end{eqnarray*}
(In both $C(L)$ and $C(X)$, we considered the maximum norm $\|\cdot\|$).
Hence the mapping
$$
{\rm sp}\, N\ni \sum_{n\in N}a_nn\longmapsto \sum_{n\in N} a_n\tilde n \in C(L)
$$
is well defined, and it is a linear isometry.  Also, we know that the linear span sp$(N)$ is dense in $Y$. Thus, we can uniquely extend this mapping to the
linear isometry defined on the whole $Y$ into $C(L)$, say $T: Y\hookrightarrow C(L)$. 

Finally, the operator $T$ it injective, a simple consequence of the Hahn--Banach theorem reveals that the adjoint operator $T^*: C(L)^*\to Y^*$ is
surjective. And recalling that $C(L)^*$ was separable, $Y^*$ must be separable as well by Hahn-Banach Theorem.

\smallskip
(2)$\Longrightarrow$(3) Is clear since the dual of $\ell_1$ is isometric to the non-separable space $\ell_\infty\,$.

\smallskip
(3)$\Longrightarrow$(1) Assume on the contrary that $X$ is not scattered. By Theorem \ref{sca}, there is a continuous surjection $\rho: X\twoheadrightarrow [0,1]$.
Then the adjoint mapping $\rho^{*}: C[0,1]\hookrightarrow C(X)$ is a linear isometry into. But the Banach space $C[0,1]$ is universal for the class of all separable Banach spaces by the 
Banach--Mazur theorem, see \cite[Proposition 1.5]{Pelczynski-Bessaga}. Therefore, $C(X)$ contains an isometric copy of $\ell_1$. Obtained contradiction completes the proof.
\end{proof}

The  proof of our main   Theorem \ref{MAIN}  uses  the following  remarkable fundamental result due to Namioka and Phelps \cite{np} for Banach spaces $C(X)$.

\begin{theorem}[Namioka--Phelps]\label{Namioka+Phelps} For every compact space $X$,
the Banach space $C(X)$ is Asplund if and only if $X$ is scattered.
\end{theorem}

\begin{proof}
The necessity will be proved in a more general Theorem~\ref{MAIN}. Just replace there $X$ by $L$ in the proof of implication `(1)$\Longrightarrow$(2)'.

Sufficiency. Assume that $X$ is scattered. Take any open convex set $D\subset C(X)$
and any convex continuous function $\vv: D\to\R$.
We have to show that $\vv$ is Fr\'echet differentiable at the points of a
dense subset of $D$.
So fix any open set $U\subset D$. Pick some separable subspace $Z\subset C(X)$
such that $U\cap Z$ is a nonempty set. For this $Z$, we find a separable over(sub)space
$Z\subset Y\subset C(X)$ as it is done in Proposition~\ref{g}.

As $X$ is scattered, Theorem~\ref{1} guarantees that the dual $Y^*$ is separable. Hence, $Y$ is an Asplund space by Proposition~\ref{111}. Thus, there is a point $f$ in the set $U\cap Y$ (open in $Y$)
where the restriction $\vv_{|Y}$ is Fr\'echet differentiable. Now, the conclusion of Proposition~\ref{g} says that the whole function $\vv$ is Fr\'echet differentiable at $f$.
We proved that the set of points where the function $\vv$ is Fr\'echet differentiable is dense in $D$.
Finally, Proposition~\ref{Phelps} completes the proof.
\end{proof}

\section{Asplund spaces $C_k(X)$ over Tychonoff spaces $X$}
In the proof of our Theorem \ref{MAIN} of this section  we will need also the following fact, see also \cite[Proposition 3.13]{KL1} for a different approach.
For safety reasons, we make more accurate the following two synonyms. Let $X$ be a topological vector space, $Y$ be a Banach space, and  $T: Y\to X$ be a linear mapping.
We say that $T$ is an {\it embedding} or an {\it isomorphism into} $X$ if it is injective, continuous, and the inverse $T^{-1}: T[Y]\to Y$ is continuous 
The symbol $\stackrel \circ B_Y$ means the open unit ball in $Y$.

\begin{lemma}\label{ku} Let $X$ be a Tychonoff space and assume that there is a Banach space $Y$ which is isomorphic to a subspace of $C_k(X)$.

\noindent Then there exists a compact set $K\subset X$ such that $Y$ is isomorphic to a subspace of the (Banach) space $C(K)$.
\end{lemma}

\begin{proof}
Let $S:Y\hookrightarrow C_k(X)$ be the promised isomorphism into. Then there must exist some $\ee>0$, $\alpha>0$, and a compact subset $K\subset X$ such that
\begin{equation}\label{14}
V^\ee_K \cap S[Y] \subset S[\stackrel \circ B_Y]\subset V^{\ee/\alpha}_K\,.
\end{equation}

We show that the (Banach) space $Y$ is isomorphic to a subspace of the (Banach) space $C(K)$.
Define $\tilde S:Y\to C(K)$ by
$$
\tilde S(y):= (Sy)|_K, \quad y\in Y.
$$
$\tilde S$ is (linear and) injective. Indeed, take any $y\in Y$ such that $\tilde Sy=0$. This means that
$(Sy)|_K\equiv0$. Then, for every $n\in \N$ we have
$$
n Sy\in V^\ee_K\cap S[Y] \subset S[\stackrel \circ B_Y]
$$
by (\ref{14}), and so $Sy\in S[\frac1{n} \stackrel \circ B_Y]$, and hence $y\in \frac 1{n} \stackrel \circ B_Y$, as $S$ is injective. Therefore, $y=0$.

Now, fix any $y\in \stackrel \circ B_Y$. By (\ref{14}) $Sy \in V^{\ee/\alpha}_K$, i.e. $Sy[K] \subset (-\ee/\alpha,\ee/\alpha)$, and hence $\|\tilde Sy\| <\ee/\alpha$. We thus proved that
$$
\tilde S[\stackrel \circ B_Y] \subset \frac\ee\alpha \stackrel \circ B_{C(K)}.
$$
It remains to prove that
$\ee \stackrel \circ B_{C(K)}\cap \tilde S[Y] \subset \tilde S[\stackrel \circ B_Y]$.
Fix any $y\in Y$ such that $\tilde Sy \in \varepsilon \stackrel \circ B_{C(K)}$; thus $\|\tilde Sy\| <\varepsilon$.
Hence 
$Sy\in V^\ee_K \cap S[Y]$.
By (\ref{14}), there is a $y'\in \stackrel\circ B_Y$ so that $Sy=Sy'$. But $S$ is  injective. Hence $ y=y'$ and so $y\in \stackrel \circ B_Y$.
We proved that
$\varepsilon \stackrel \circ B_{C(K)} \cap \tilde S[Y] \subset \tilde S[\stackrel \circ B_Y],
$
and therefore, $\tilde S$ is an isomorphism of $Y$ onto $\tilde S[Y] \  \ (\subset C(K)$). Summarizing, we obtained that
$$
\ee \stackrel \circ B_{C(K)}\cap \tilde S[Y] \subset \tilde S[\stackrel \circ B_Y] \subset \frac\ee\alpha \stackrel \circ B_{C(K)};
$$
that is, analytically (using the injectivity of $\tilde S$),
$$
\forall y\in Y\quad \ee\|y\|_Y \le  \|\tilde Sy\|_{C(K)} \le \frac\ee\alpha \|y\|_Y.
$$  \end{proof}

M. Komisarchik and M. Megrelishvili  \cite{KM} say that a lcs $E$ has the {\it Namioka--Phelps property}  if  every bounded set $B\subset E$ is \emph{fragmented} on each weak$^{*}$-compact equicontinuous subset $K$ of the dual $E^*$ of $E$; that is, for each nonempty subset $A\subset K$ and each $\varepsilon >0$, there exists a weak$^{*}$-open subset $U\subset E^*$ such that $U\cap A\neq\emptyset$ and ${\rm diam}\{\langle v,x\rangle:\ x\in U\cap A,\ v\in B\}<\varepsilon.$

The  equivalence `(1)$\Longleftrightarrow$(2)' in the  main Theorem \ref{MAIN} formulated below has been proved in \cite[Corollary 3.8]{es} as a consequence of several partial facts. We believe that our approach is mostly self-contained.
The equivalences `(2)$\Longleftrightarrow$(3)$\Longleftrightarrow$(4)' were proved  in \cite{GKKM}.
Finally, the equivalencies `(2)$\Longleftrightarrow$(5)$\Longleftrightarrow$(6)' were proved in \cite{KM}.

\begin{theorem}\label{MAIN}
Let $X$ be a Tychonoff space. The following assertions are equivalent:
\begin{enumerate}
\item   $C_k(X)$ is an Asplund space.
\item   Every compact set in $X$ is scattered.
\item  The space $C_k(X)$ does not contain any isomorphic copy of $\ell_{1}$.
\item Every separable Banach subspace of $C_k(X)$ has separable dual.
\item $C_k(X)$ satisfies the Namioka--Phelps property.
\item $C_k(X)$ is tame, i.e. every weak$^{*}$-compact equicontinuous convex subset $K\subset C_k(X)^*$ is the closed (in the strong  topology) convex hull of the extreme points of $K$.
\end{enumerate}
\end{theorem}
\begin{proof}
(1)$\Longrightarrow$(2)
  Assume that $C_k(X)$ is an Asplund space and that some compact subset set $L\subset X$ is not scattered. Find then a nonempty closed perfect set $P\subset L$, by
Lemma~\ref{scattered}.
Consider the function $\vv: C_k(X)\to\R$ defined by
$$
C_k(X)\ni f\longmapsto \max f(P)=:\vv(f).
$$
Clearly, $\vv$ is convex. It is also continuous on $C_k(X)$. Indeed, pick any $f\in C_k(x)$ and any $\ee>0$.
Then
for every $h\in V_P^\ee$ we have
$$
\vv(f+h)-\vv(f)\le \max f(P)+\max h(P)-\max f(P) < \ee
$$ and, similarly,
$$\vv(f)-\vv(f+h)\le \max (f+h)(P)+\max (-h)(P)-\max (f+h)(P) < \ee;
$$
thus $|\vv(f+h)-\vv(f)| <\ee$.

Put
$$
M:=\{h\in C_k(X):\ h(X) \subset [-1,1]\};
$$
this is obviously a bounded set.
As $C_k(X)$ is Asplund and $\vv$ is continuous convex,
$\vv$ is Fr\'echet differentiable at some $f\in C_k(X)$.
Thus, in particular, there is a linear continuous $\xi: C_k(X)\rightarrow\R$ such that
$$
\sup_{h\in M}\Big|\frac{(\vv(f+th)-\vv(f)}t - \xi(h)\Big|\longrightarrow 0\quad {\rm as}\quad t\to0;
$$
and hence (getting rid of the almost redundant $\xi$)
$$
(0 \le )\ \sup_{h\in M}\frac{(\vv(f+th)+\vv(f-th)-2\vv(f)}t\longrightarrow 0\quad {\rm as}\quad t\to0;
$$
Thus, in particular, (for $\ee:=\frac12$) there is a $\delta>0$ such that
$$
\vv(f+\delta h)+\vv(f-\delta h)-2\vv(f) < \frac\delta2
$$
for every $h\in M$.

Now, pick a $p\in P$ so that $f(p)=\vv(f)$. As $P$ is perfect, there is a $q\in P\setminus \{p\}$ such that $f(q)> \vv(f)-\frac\delta 2$.
We recall that $X$ is a Tychonoff space and $L$ a compact set in it. Hence, there is a continuous function $h: X\to [0,1]$ such that
$h(p)=0$ and $h(q)=1$. Note that then $\vv(h)=1$ and $h\in M$.
Thus
\begin{eqnarray*}
\delta & =& \delta h(q)-\delta h(p) = (f+\delta h)(q) + (f-\delta h)(p) -f(q)-f(p) \cr
& <& \vv(f+\delta h) + \vv (f-\delta h) - 2\vv(f) +\frac\delta 2 < \frac\delta 2 + \frac\delta 2 =\delta
\end{eqnarray*}
(Note that we used a \v Smulyan like argument, see \cite{dgz}) and \cite[p. 342--344]{y2}; a contradiction. We proved that $L$ must be scattered.
 \smallskip

(2)$\Longrightarrow$(1)  Here we freely imitate arguments from \cite{s} and \cite{es}. Let $D\subset C_k(X)$ be any (nonempty) open convex set and consider any convex function $\vv: D\to \R$, everywhere continuous in the $k$-topology of the space $C_k(X)$.
We want to show that the set of points where $\vv$ is Fr\'echet differentiable contains a dense $G_\delta$ subset of $D$.



Pick some $u\in D$. The continuity of $\vv$ at $u$ yields an $\ee>0$ and a compact subset $L\subset X$ such that $u+V^\ee_L \subset D$ and that $\vv$ is bounded above on the set $u+V^{\ee}_L$.
A simple picture, with \cite[p. 4, 5]{p}, then reveals that
for every $v\in D$ there is a $\gamma>0$ such that $v+ V^{\gamma}_L \subset D$ and that
$\vv$ is bounded above on $v+ V^{\gamma}_L$.

\smallskip

Let $C(L)$ be the Banach space of all continuous functions on $L$, with `maximum' norm. Its closed unit ball (around the origin) is denoted by $B_{C(L)}$. Consider the 'restriction' mapping $T: C_k(X)\to C(L)$ defined by
$$
C_k(X)\ni f\longmapsto f|_L=:Tf \in C(L).
$$
Clearly, $T$ is linear. It is also continuous since, for every $\gamma>0$, we have $T[V^\gamma_L]\subset
\gamma\!\!\stackrel \circ B_{C(L)}$. (We have here even equality by Proposition~\ref{T}). Since $T$ is continuous, $T^{-1}(\Omega)$ is open or $G_\delta$
whenever $\Omega\subset C(L)$ is, respectively, open or $G_\delta$. Let us further check that $T[D]$ is an open subset of $C(L)$.
So, fix any $f\in D$. Find a $\gamma>0$ so that $f+ V^\gamma_L\subset D$. Then
$$
Tf +\gamma\!\!\stackrel \circ B_{C(L)} = T[f+ V^\gamma_L] \subset T[D];
$$
the equality here comes from Proposition~\ref{T}, as $X$ is Tychonoff and $L$ is compact.
We verified that $T[D]$ is open.

We further observe that, given any dense subset $\Omega$
in $T[D]$, then $T^{-1}[\Omega]$ is dense in $D$. Indeed,
pick any $v\in D$, any $\beta>0$, and any compact set $L'$ in $X$. We want to show that $(v+V^\beta_{L'})\cap T^{-1}
[\Omega]\neq\emptyset$. Take $\gamma\in(0,\beta)$ so small that $v+V^{\gamma}_L\subset D$.
As $\Omega$ is dense in $T[D]$, we have that
$$
(Tv+\gamma\!\! \stackrel \circ B_{C(L)})\cap \Omega\neq\emptyset;
$$
hence, there is an $h\in \gamma\!\! \stackrel \circ B_{C(L)}$ so that $Tv+h\in\Omega$.
By Proposition~\ref{T}, 
we find a continuous function $\check h:X\to\R$ such that
$\check h[X] \subset (-\gamma,\gamma)$ and $\check h|_L=h$. Then $v+\check h \in v + V^{\gamma}_{L'} \subset v + V^{\beta}_{L'}$
and $T(v+\check h)=Tv + h\ (\in\Omega)$. Thus $$v+\check h\in (v+V^\beta_{L'})\cap T^{-1}[\Omega].$$
We proved that $T^{-1}[\Omega]$ is dense in $D$; note that $T^{-1}[\Omega]\cap D$ {\sl is dense in} $D$, too.

Now, we define the function $\psi: T[D]\to\R$ as
$$
\psi(g):=\vv(\check g),\quad g\in T[D],
$$
where $\check g$ is {\tt any} element of $D$ 
such that $\check g|_L=g$.
This is a correct definition because, 
if $f,f'\in D$ and $f|_L=f'|_L=g$, then by the Claim from the proof of Proposition~\ref{SY}, $\vv(f)=\vv(f')$.

The function $\psi$ is convex. Indeed, take any $g_1,g_2\in {C(L)}$ and any $\alpha\in (0,1)$
Find some $\check{g_1}, \check{g_2}\in D$, with $\check{g_1}|_L=g_1$ and $\check{g_2}|_L=g_2$.
Then 
$$\big(\alpha \check{g_1} +(1-\alpha)\check{g_2}\big)|_L= \alpha g_1+(1-\alpha)g_2,$$ and
hence
\begin{eqnarray*}
\psi\big(\alpha g_1+(1-\alpha)g_2\big) &=& \vv \big(\alpha \check{g_1} +(1-\alpha)\check{g_2}\big) \le \alpha \vv(\check{g_1}) + (1-\alpha) \vv(\check{g_2}\big)\cr
&=&\alpha \psi({g_1}) + (1-\alpha) \psi({g_2}).
\end{eqnarray*}

Let us check that our $\psi$ is continuous on $T[D]$. Fix any $f\in D$. We know that there is an $\ee>0$ so that $f+V^\ee_L$ lies in $D$ and that
$\vv[f+ V^\ee_L]$ is bounded above.
Thus the set $$\psi[Tf+\ee\!\!\stackrel \circ B_{C(L)}] = \vv[f+V^\ee_L]$$ is also bounded above (we needed Proposition~\ref{T}). And having this, then
\cite[p. 4, 5]{p} (valid for Banach spaces) guarantees the continuity of $\psi$.

Now, since $L$ is a scattered compact set by the assumption,  Theorem \ref{Namioka+Phelps}  says that
the Banach space $C(L)$ is Asplund. Hence our $\psi$ is Fr\'echet differentiable at the points of
a dense $G_\delta$ subset, say $\Omega$, of $T[D]$.
Fix any $f$ in $T^{-1}(\Omega)\cap D$ (which is already known to be a dense $G_\delta$ subset of $D$).
Find an $\ee>0$ so small that $f+V^\ee_L\subset D$ (we already proved that such an $\ee$ exists).
As $Tf\in\Omega$, the function $\psi$ is Fr\'echet differentiable at $Tf$, with derivative, $\xi: C(L)\to\R$, say.
Having this we get
\begin{eqnarray*}
&&\sup_{h\in V_L^\ee}\big|\vv(f+ th)-\vv(f) - \xi(th|_L)\big| \cr
&= & \sup_{z\in \ee \stackrel \circ B_{C(L)}}\big( \psi(T f + tz)-\psi(T f)-\xi(tz)\big) = o(t)\,\,{\rm as}\,\,(-\ee,\ee)\ni t\to 0;
\end{eqnarray*}
Here we again used Proposition~\ref{T}. Therefore, our $\vv$ is Fr\'echet differentiable at $f$, with derivative $$C_k(X)\in h \longmapsto \xi(h|_L)=: \vv'(f)h.$$
Actually we should in addition check that this $\vv'(f)$ is linear and continuous. Indeed, we have that
$\xi[V^\ee_L] \subset (-\ee,\ee)$.
\smallskip



(3)$\Longrightarrow$(2)
Assume that $X$ contains a compact subset $K$ which is not scattered. By Theorem~\ref{sca}, there exists a continuous surjection $\rho: K\twoheadrightarrow [0,1]$.
Using Proposition~\ref{T}, $\rho$ can be extended to a continuous surjection, say  $\check \rho: X\twoheadrightarrow [0,1]$.
The natural adjoint mapping to $\check\rho$, say $T: C[0,1]\hookrightarrow C_k(X)$, is then an  embedding (i.e. an isomorphism onto its range). Indeed, a bit of routine calculation  reveals that
$$
V^1_K \cap T[C[0,1]]= T\big[\stackrel \circ B_{C[0,1]}\big]\  \ (\subset V^1_L\ \ {\rm for \ every\  compact}\  \ L\subset X).
$$
(We note that the above equality says that our $T$ is a `bound-covering mapping', see \cite{es}).
Since $C[0,1]$ is universal in the class of all separable Banach spaces, by \cite[Proposition 1.5]{Pelczynski-Bessaga}, 
the Banach space $\ell_{1}$ (isometrically) embeds into $C[0,1]$. Hence $\ell_1$  embeds into $C_k(X)$, which is a negation of (3).
\smallskip

(2)$\Longrightarrow$(3)  Let (2) hold. Assume that $C_{k}(X)$ contains an isomorphic copy of $\ell_{1}$. By Lemma \ref{ku} there exists 
a compact (scattered) subset $K\subset X$ such that $\ell_1$ is isomorphic to a subspace of the Banach space $C(K)$. But his contradicts Theorem \ref{1}.
Hence $C_k(X)$ does not contain any isomorphic  copy of $\ell_1$.

\smallskip
(4)$\Longrightarrow$(3) It is true since the dual of $\ell_1$ is isometric to $\ell_\infty$, which is non-separable.

\smallskip

(2)$\Longrightarrow$(4) It follows from Theorem~\ref{1} and Lemma~\ref{ku}.
\end{proof}

\begin{remark} (a) Statement (1) in Theorem ~\ref{MAIN} can be strengthened to:

\noindent `{\sl For every Asplund Banach space $E$ the product $C_k(X)\times E$ is also Asplund.}'

\noindent Indeed, in the proof of `(2)$\Longrightarrow$(1)', it is enough to replace, in suitable places, `$C_k(X)$' by `$C_k(X)\times E$'.
More specifically, the mapping $C_k(X)\times E\ni(f,g)\longmapsto (f|_L, g)\in C(L)\times E$ is still `'bound covering' (as it was the mapping $T$). And the latter product (of two Banach spaces) is Asplund by \cite{np}.
 Another way how to check this fact is by using the paper \cite{es}; 

\noindent (b) From another hand, if both spaces $X$ and $Y$ have all compact subsets scattered, respectively, 
then $C_k(X)\times C_k(Y)$ is Asplund. Indeed, it is well known that $C_k(X)\times C_k(Y)$ is isomorphic to $C_k(X\oplus Y)$, 
and then in the direct sum $X\oplus $Y all compact subsets are scattered and Theorem \ref{MAIN} applies.
\end{remark}

\begin{problem}
Assume that $C_k(X)$ is an Asplund space. Is the product $C_k(X)\times E$ Asplund for every Asplund lcs $E$?
\end{problem}

Analyzing the proof of `(2)$\Longrightarrow$(1)' in Theorem \ref{MAIN},
 we can deduce  the following fact related with Theorem \ref{Phelps},  see also \cite[Theorem 3.2, Example 3.6]{es}.

\begin{theorem}\label{first}
If $X$ is a Tychonoff space and $\vv$ is a continuous convex function on a nonempty  convex open  $D\subset C_k(X)$,
the (possibly empty) set of points of Fr\'echet differentiability of $\vv$ is a $G_{\delta}$.
\end{theorem}

\section{$\Delta_1$-spaces $X$ and the Asplund property for spaces  $C_k(X)$}

The aim of this section is to sketch briefly a proof of Theorem \ref{mainIII} which originally appeared in a bit stronger form as \cite[Theorem 3.2]{KKuL}. 
The class of $\Delta_1$-spaces introduced and explored in \cite{KKuL} plays here a key role.
For any $\Delta_1$-space $X$, each compact subset of $X$ is scattered, and hence $C_k(X)$ is an Asplund lcs.
Similarly, we obtain another equivalent topological description of a compact space $X$ such that the Banach space $C(X)$ is Asplund.

Let $X$ be a topological space. The following well-known topological notions generalize compactness.
\begin{itemize}
\item $X$ is called {\it $\omega$-bounded} if the closure of every countable subset in $X$ is compact.
\item $X$ is called {\it countably compact} if every countably infinite subset of $X$ has a limit point.
\item $X$ is called {\it pseudocompact} if every continuous real-valued function on $X$ is bounded.
\end{itemize}

Clearly, $X$ is $\omega$-bounded implies that $X$ is countably compact, and the latter implies pseudocompactness of $X$.
For the various examples distinguishing these and some other classes of topological spaces the reader is advised to look \cite{KKuL}.
Here we mention only that a very concrete example of a locally compact non-compact space  which is $\omega$-bounded and scattered, 
is the interval of countable ordinals $[0,\omega_1)$ with the order topology.
The Tychonoff plank $X=[0,\omega_1]\times [0,\omega]\setminus\{(\omega_1,\omega)\}$ is a pseudocompact scattered space which is not $\omega$-bounded.

\begin{theorem} \cite{KKuL} \label{mainIII}
For a Tychonoff $\omega$-bounded space $X$ the following assertions are equivalent:
\begin{enumerate}
\item $X$ is scattered.
\item Every compact subset of $X$ is scattered.
\item $X$ is a $\Delta_1$-space.
\item $C_k(X)$ is an Asplund space.
\end{enumerate}
\end{theorem}

\begin{definition} \cite{KKuL1}, \cite{KL1}
A topological  space $X$ is called a $\Delta$-space  ($\Delta_1$-space) if for every decreasing sequence $(D_n)_{n\in\omega}$
of (countable, respectively) subsets of $X$ with $\bigcap_{n\in\omega}D_n=\emptyset$, there is a decreasing sequence of open sets $(V_n)_{n\in\omega}$ such that $D_n \subset V_n$ 
for each $n\in\omega$,
and again with $\bigcap_{n\in\omega}V_n=\emptyset$.
\end{definition}

The original definition of a $\Delta$-set of the real line $\R$ is due to M. Reed and E. van Douwen, see \cite{Reed}.
The $\Delta$- and $\Delta_1$-properties turned out to have nice applications in the theory of function spaces $C(X)$ endowed with the topology
of pointwise convergence, see \cite{KKuL1}, \cite{KL1}. 

Another motivation for studying the $\Delta_1$-property is provided by the fact that it extends the classical 
notion of $\lambda$-sets. Recall that $X\subset\R$ is called a $\lambda$-set if every countable $A \subset X$ is a $G_{\delta}$-subset of $X$ and, 
more generally, a topological space $X$ is a $\lambda$-space if every countable subset $A \subset X$ is a $G_{\delta}$. The study 
of $\lambda$-sets dates back to 1933 when K. Kuratowski proved in ZFC (see \cite[Chapter 3, §40.III]{k}) that there exist uncountable $\lambda$-sets. 
Since every $\lambda$-space has the $\Delta_1$-property, the notion of a $\Delta_1$-space determines a wider class than the class 
of $\lambda$-spaces. It is also worth mentioning that a metrizable space has the $\Delta_1$-property if and only if it is a 
$\lambda$-space according to \cite[Theorem 2.19]{KKuL1}.

First we present a proof of the equivalence (1) and (2) in Theorem \ref{mainIII}.
\begin{lemma}\label{ten}
Let $X$ be a Tychonoff $\omega$-bounded space. Then every compact subset of $X$ is scattered 
if and only if  $X$ is scattered.
\end{lemma}

\begin{proof}
 Assume on the contrary that each compact subset in $X$ is scattered but $X$ is not scattered. Since the space  $X$ is pseudocompact, we can apply \cite[Proposition 5.5]{Leiderman-Tka} to get in $X$  a closed subset $F$ for which there exists  a  continuous surjection $h: F\twoheadrightarrow [0,1].$  Let $Q$ be a dense countable subset of $[0,1]$. Choose a countable subset $C\subset F$ such that $h(C)=Q$. 
 Let $A:=\overline{C}$ be the closure in $F$. Note that $F$ is also $\omega$-bounded. Therefore, $A$ is compact and obviously $h(A)=[0,1]$. 
By assumption, the compact set $A$ is scattered, this is a  contradiction, in view of Theorem \ref{sca}. Hence $X$ is scattered. The converse is clear.
\end{proof}

We refer the reader to the following important results concerning $\Delta_1$-spaces.

\begin{theorem}\label{del-1}\hfill
\begin{itemize}
\item [(a)] \cite[Theorem 2.15]{KKuL1}
Every regular scattered space is a $\Delta_1$-space.
\item [(b)]  \cite[Corollary 2.16]{KKuL1}
A compact space $X$ is a $\Delta_1$-space if and only if $X$ is scattered.
\item [(c)]  \cite[Theorem 2.9]{KKuL1}
A pseudocompact space $X$ is a $\Delta_1$-space if and only if every countable subset of $X$ is scattered.
\item [(d)]  \cite[Theorem 3.9]{KKuL1}
Let $X$ be a countable union of closed subsets $X_n$. If each $X_n$ is a $\Delta_1$-space 
then $X$ also is a $\Delta_1$-space.
\end{itemize}
\end{theorem}

We are ready to prove our main theorem of this section.

\begin{proof}[Proof of Theorem \ref{mainIII}]

(1)$\Longleftrightarrow$(2) This is Lemma \ref{ten}.

(1)$\Longrightarrow$(3)  Trivially follows from Theorem \ref{del-1}(a).  

(3)$\Longrightarrow$(2) Assume  that $X$ is a $\Delta_1$-space. Then every compact subset of $X$ is a $\Delta_1$-space. 
We apply Theorem \ref{del-1}(b).
\end{proof}

\begin{remark} The statement of Theorem \ref{mainIII} is not true if we assume only that $X$ is countably compact.
Indeed, there exists a countably compact non-scattered dense subspace $X$ in the reminder of the Stone-\v{C}ech compactification $\beta \omega\setminus \omega$ 
such that all countable subsets of $X$ are scattered, hence $X$ is a $\Delta_1$-space. Moreover, every compact subset of $X$ is finite, see \cite[Example 2.18]{KKuL1}.
\end{remark}

\begin{remark} There are various examples of Tychonoff spaces $X$ such that every compact subset of $X$ is scattered but $X$ is not a $\Delta$-space.
A {\it Bernstein set} $\B$ is a subset of the real line such that both $\B$ and $\R\setminus\B$ intersect with every uncountable closed subset of the real line.
It follows that every compact subset of $\B$ is countable. However, $\B$ is not a $\Delta_1$-space, see \cite[Example 2.20]{KKuL1}. 

Also, there is a pseudocompact space $X$ such that every compact subset of $X$ is scattered, but $X$ is not a $\Delta_1$-space, see \cite[Example 3.2]{KKuL1}. 
\end{remark}

We  complete this section  again with the Namioka--Phelps's theorem which is actually extended  by  an equivalent condition (2).
\begin{theorem}
For a compact space $X$ the following assertions are equivalent:
\begin{enumerate}
\item $X$ is scattered.
\item $X$ is a $\Delta_1$-space.
\item The Banach space $C(X)$ is Asplund.
\end{enumerate}
\end{theorem}

\section{Final remarks and illustrating examples}

It is well known (but not so easy to prove) that every closed vector subspace of any Asplund Banach space is again Asplund, see for example  \cite{p}, \cite[Proposition 2.33]{np}.  We show however
that this property fails in general lcs, see \cite{Kakol-Leiderman}.
First recall that $\varphi$ denotes an $\aleph_{0}$-dimensional vector space endowed with the finest locally convex topology, say $\xi_\infty$, that is, the topology, where
a base of the origin consists of all absolutely convex absorbing sets. Then, clearly, every linear functional on $\varphi$ is continuous.
 
Every bounded set $B$  in $\varphi$ is finite-dimensional, i.e. the linear span of $B$ is finite-dimensional, because otherwise (applying the Hahn--Banach theorem) one can  construct a linear discontinuous  functional on $\varphi$. On the other hand, $\varphi$  is the strict countable inductive limit of finite-dimensional Banach vector subspaces $\R^{n}$.
Indeed, consider on $\varphi$ the locally convex inductive limit topology $\xi$ generated by finite-dimensional linear subspaces, see \cite[Chapter 2.4]{KKPS} or \cite[0.3.1]{Bonet}.
Then, since the original topology $\xi_{\infty}$ of $\varphi$ is the finest locally convex topology on $\varphi$, the both topologies $\xi$ and $\xi_{\infty}$ coincide. Moreover, the topology 
$\xi=\xi_{\infty}$ is complete, see  \cite[Proposition 8.4.16 (iv)]{Bonet},  or \cite[Proposition 1]{Zelazko} for a direct proof.

B. Sharp \cite{s} observed that the space $\varphi$ is not GDS.
To show this, consider the $l_1$ norm $\|\cdot\|_{1}$ on it. This is a convex function, continuous with respect to the original topology $\xi_\infty$ on $\varphi$.
And it is easy to check that $\|\cdot\|_{1}$ is nowhere  G\^ ateaux differentiable on $\varphi$ [32, Example 1.4. (b)].
Indeed, given any $x\in \varphi$,
there is an $n\in \N$ such that the $n$-th coordinate of $x$ is 0.
Let $e_n:= (0, 0, \dots, 1, 0, \dots)$ be the sequence whose only nonzero term is equal to 1 on the $n$-th place.
Then
$$
\lim_{t\downarrow0}(\|x+t e_n\|_1-\|x\|_1)/t= 1\quad {\rm and}\quad \lim_{t\uparrow0}(\|x+t e_n\|_1-\|x\|_1)/t= -1,
$$
and hence, $\|\cdot\|_1$ is not  G\^ ateaux differentiable at the point $x$.

\begin{remark}
Note that from the description $\varphi$  as the the strict countable inductive limit of finite-dimensional linear subspaces $\R^{n}$  it follows that the lcs $\varphi$ can be identified with the free locally convex space $L(\N)$. It has been shown in \cite{Kakol-Leiderman} that the free lcs $L(X)$ is not GDS for every infinite Tychonoff space $X$.
\end{remark}

We complete this section with  the following interesting example mentioned in \cite{Kakol-Leiderman}. Below $Q$  denotes the set of all rational numbers with the topology induced from $\R$. 
So, $Q$ is a metric countable space.

\begin{proposition}\cite {s}, \cite{Kakol-Leiderman}
The space $C_k(Q)$ over the rationals $Q$ is an Asplund space but contains isomorphically a closed copy of the non-Asplund space $\varphi$.
\end{proposition}
\begin{proof}
Every compact subset of $Q$ is countable, hence it is scattered. By Theorem \ref{MAIN} we deduce that $C_k(Q)$ is an Asplund space. Since $Q$ is a $\mu$-space, the space $C_k(Q)$ is barrelled by the Nachbin--Shirota theorem, see \cite[Propostion 2.4.16]{KKPS}, i.e., every closed absolutely convex absorbing subset of $C_k(Q)$ is a neighbourhood of zero.

On the other hand, the space $C_k(Q)$ is not Baire-like, i.e., there exists an increasing sequence of closed convex absorbing subsets $A_1, A_2,\ldots$ covering the space $C_k(Q)$ such that no $A_n$ is a neighbourhood of the origin in $C_k(Q)$. Indeed, assume on the contrary that $C_k(Q)$ is a Baire-like space. Then we follow the argument contained in the proof of \cite[Proposition 2.4.20]{KKPS}, 
or \cite[Corollary 5.3.4]{mccoy}: For $n\in\N$ we put
$$
U_n:= (-1/n,1/n)\cap Q\quad {\rm and}\quad H_n:= \big\{f\in C(Q):\  f[U_n]\subset [-n,n]\big\}.
$$
Clearly, the $U_n$'s form a basis of neighbourhoods of $0$ in $Q$ and the (closed convex and absorbing) sets $H_n$'s cover all of $C_k(Q)$.
Also, clearly, $U_1 \supset U_2\supset \cdots$ and $H_1\subset H_2\subset \cdots$.

Since we assumed that $C_{k}(Q)$ is Baire-like, there exist a compact
set $K\subset Q$, $\varepsilon>0$, and $n\in\N$ such that
$$
\{f\in C(X):\ f[K] \subset (-\varepsilon,\varepsilon)\}\subset H_{n}.
$$
Note that $U_{n}\subset K$. Indeed, if there is a $z\in U_{n}\setminus K$, then putting
$$
f(x):= 2n\frac{{\rm dist}(x,K)}{{\rm dist}(z,K)},\quad x\in Q,
$$
we have that $f\in C_k(Q)$ and that $f(z)=2n$. Also $f_{|K}\equiv 0$, and thus $f\in U_n$, which is impossible.
Therefore, we have that $U_{n}\subset K$. This leads to a contradiction because $K$ is a countable compact subset of $\R$ and the closure of $U_{n}$ in $\R$ is an uncountable closed interval. 
We have proved that the space is $C_k(Q)$ is not Baire-like.
Finally, since $C_k(Q)$ is barrelled but not Baire-like, we apply the Saxon's result (see \cite[Corollary 2.4.4]{KKPS}), and conclude that lcs $C_k(Q)$ contains isomorphically a (closed) copy of $\varphi$.
\end{proof}

\begin{remark}  The locally convex space $\varphi$ is nonmetrizable, and it is complete as the strict countable inductive limit of finite-dimensional spaces. 
From another hand, its  strong dual $\varphi^{*}$, endowed with the uniform convergence topology on bounded sets from $\varphi$ is isomorphic to the countable power ${\R}^{\omega}$,
 hence the strong dual $\varphi^{*}$ is separable.
However, $\varphi$ is not an Asplund space.

This shows that the Asplund property for lcs differs dramatically from the Asplund property in the class of Banach spaces.
\end{remark}

\end{document}